\documentclass[10pt]{article}
 \font\smallit=cmti10

\usepackage{amssymb,latexsym,amsmath,epsfig,amsthm} 

\makeatletter

\renewcommand{\@seccntformat}[1]{\csname the#1\endcsname. }
\makeatother

\newtheorem{theorem}{Theorem}[section]
 \newtheorem{lemma}[theorem]{Lemma}

 \newtheorem{corollary}[theorem]{Corollary}
 \newtheorem{definition}[theorem]{Definition}
 \newtheorem{remark}[theorem]{Remark}
\begin{document}

\begin{center}
 {\bf Combinatorial Sums $\sum_{k\equiv r(\mbox{mod }
m)}{n\choose k}a^k$ and Lucas Quotients (II)}
 \vskip 20pt
 {\bf Jiangshuai Yang}\\
 {\smallit Key Laboratory of Mathematics Mechanization, NCMIS, Academy of Mathematics and Systems Science, Chinese Academy of Sciences, Beijing 100190, People's Republic of China}\\
 {\tt yangjiangshuai@amss.ac.cn}\\
 \vskip 10pt
 {\bf Yingpu Deng}\\
 {\smallit Key Laboratory of Mathematics Mechanization, NCMIS, Academy of Mathematics and Systems Science, Chinese Academy of Sciences, Beijing 100190, People's Republic of China}\\
 {\tt dengyp@amss.ac.cn}\\
 \end{center}
\vskip 30pt
\centerline{\bf Abstract}

\noindent In \cite{dy}, we obtained some congruences for Lucas quotients of two infinite families of Lucas sequences by studying the combinatorial sum
$$\sum_{k\equiv r(\mbox{mod }m)}{n\choose k}a^k.$$
In this paper, we  show that the sum can be expressed in terms of some recurrent sequences with orders not exceeding $\varphi{(m)}$ and
  give some new congruences.

\pagestyle{myheadings}
 \thispagestyle{empty}
 \baselineskip=12.875pt
 \vskip 30pt

\section{Introduction}

\noindent Let $p$ be an odd prime, using the formula for the  sum
$$\sum_{k\equiv r(\mbox{mod }8)}{n\choose k},$$
   Sun \cite{s1995} proved that
\[\sum\limits_{k=1}^{\frac{p-1}{2}}\frac{1}{k\cdot2^k}\equiv\sum\limits_{k=1} ^{[\frac{3p}{4}]}\frac{(-1)^{k-1}}{k}\pmod p.\]
Later, Shan and E.T.H.Wang \cite{sw} gave a simple proof of the above congruence. In \cite{sun5}, Sun proved five similar congruences  by using  the formulas for Fibonacci quotient and Pell quotient.

In \cite{s2002}, Sun  showed that the sum
 $$\sum_{k\equiv r(\mbox{mod }m)}{n\choose k},$$
 where $n,m$ and $r$ are integers with $m,n>0$, can be expressed in terms of some recurrent sequences with orders not exceeding $\varphi{(m)}/2$, and obtained the following congruence
 \[\sum_{k=1}^{\frac{p-1}{2}}\frac{3^k}{k}\equiv\sum_{k=1}^{\left[\frac{p}{6}\right]}\frac{(-1)^k}{k} \pmod p.
\]

In \cite{dy},  we studied more general sum
\begin{equation}\label{generalsum}
\sum_{k\equiv r(\mbox{mod }m)}{n\choose k}a^k,
\end{equation}
and obtained congruences for Lucas quotients of two infinite families of Lucas sequences. See (\cite{dy} Theorems  4.10 and 5.4). In this paper, we continue studying the sum. We show that it can be expressed in terms of some recurrent sequences with orders not exceeding $\varphi{(m)}$, and
 obtain some new congruences.

  For $x\in\mathbb{R}$, we use $[x]$ to denote the integral part of $x$ i.e., the largest integer $\leq x$. For odd prime $p$ and integer $b$, let $\left(\frac bp\right)$ denote the Legendre symbol   and $q_p(b)$ denote the Fermat quotient $(b^{p-1}-1)/p$ if $p\nmid b$.   When $c,d\in\mathbb{Z}$, as usual $(c,d)$ stands for the greatest common divisor of $c$ and $d$.  For any positive integer $m$, let $\zeta_m=e^{\frac{2\pi i}{m}}$ be the primitive $m$-th root of unity and let $\varphi({m})$,  $\mu{(m)}$ denote the Euler totient function and   M$\ddot{\textup{o}}$bius  function respectively.   Throughout this paper, we fix $a\neq 0,\pm1$.

  \section{Main Results}

\begin{definition}\label{defsum}
{\rm Let $n,m,r$  be integers with $n>0$ and $m>0$. We define
  $$\left[\begin{array}{c}n \\ r\\\end{array}\right] _{m}(a):=\sum_{\substack{k=0\\k\equiv r({\mbox{mod }}m)}}^n\binom nk a^k,$$
  where  ${n\choose k}$ is the binomial coefficient with the convention ${n\choose k}=0$ for $k<0$ or $k>n$.}
\end{definition}
 \noindent Then we have the following theorem.
  \begin{theorem}\label{Maintheorem}
  Let $m,n\in\mathbb{Z}^+$, and $k\in\mathbb{Z}$. Write
  $$W_{n}(k,m)=\sum_{\substack{l=1\\(l,m)=1}}^m\zeta_m^{-kl}(1+a\zeta_m^l)^n,$$
  and
   $$A_{m}(x)=\prod_{\substack{l=1\\(l,m)=1}}^m(x-1-a\zeta_m^l)=\sum\limits_{s=0}^{\varphi(m)}b_sx^s.$$
   Then
  $$A_{m}(x)\in \mathbb{Z}[x] \quad and \quad\sum\limits_{s=0}^{\varphi(m)}b_sW_{n+s}(k,m)=0.$$
  Moreover, for any   $r\in\mathbb{Z}$ we have
  \begin{equation*}
  \left[\begin{array}{c}n \\r\\\end{array}\right] _{m}(a)=\frac1m\sum\limits_{d\mid m }W_{n}(r,d).
  \end{equation*}
 \end{theorem}
  \begin{proof}
    It is easy to see that the coefficients of $A_{m}(x+1)$ are symmetric polynomials in those primitive $m$-th roots of unity with integer coefficients. Sicne
    \[\Phi_m(x)=\prod_{\substack{l=1\\(l,m)=1}}^m(x-\zeta_m^l)\in\mathbb{Z}[x],\]
     $A_{m}(x+1)\in \mathbb{Z}[x]$ by Fundamental Theorem on Symmetric Polynomials. Therefore $A_{m}(x)\in \mathbb{Z}[x].$

    For any positive integer $n$, we clearly have
    \begin{align*}
    \sum\limits_{s=0}^{\varphi(m)}b_sW_{n+s}(k,m)
    &=\sum\limits_{s=0}^{\varphi(m)}b_s\sum_{\substack{l=1\\(l,m)=1}}^m\zeta_m^{-kl}(1+a\zeta_m^l)^{n+s}\\
    &=\sum_{\substack{l=1\\(l,m)=1}}^m\zeta_m^{-kl}(1+a\zeta_m^l)^{n}\sum\limits_{s=0}^{\varphi(m)}b_s(1+a\zeta_m^l)^{ s}\\
    &=\sum_{\substack{l=1\\(l,m)=1}}^m\zeta_m^{-kl}(1+a\zeta_m^l)^{n}A_{m}(1+a\zeta_m^l)\\
    &=0.
    \end{align*}
    Let $r\in\mathbb{Z}$, then we have
    \begin{align*}
      \left[\begin{array}{c}n \\ r\\\end{array}\right] _{m}(a)
      &=\sum\limits_{k=0}^n\binom nka^k\cdot\frac1m\sum\limits_{l=1}^m\zeta_m^{(k-r)l}\\
      &=\frac1m\sum\limits_{l=1}^m\zeta_m^{-rl}(1+a\zeta_m^l)^n\\
      &=\frac1m\sum\limits_{d\mid m}\sum_{\substack{b=1\\(b,d)=1}}^d\zeta_d^{-rb}(1+a\zeta_d^b)^n\\
      &=\frac1m\sum\limits_{d\mid m}W_{n}(r,d).
    \end{align*}
    This ends the proof.
  \end{proof}
  Note that the theorem   is a generalization of Theorem 1 of \cite{s2002}.
 \begin{remark}\label{Wremark}
  The last result shows that $\left[\begin{array}{c}n \\ r\\\end{array}\right] _{m}(a)$ can be expressed in terms of some linearly recurrent sequences with orders not exceeding $\varphi{(m)}.$
 \end{remark}
  Now we list $A_{m}(x)$ for $1\leq m\leq6$:
 \begin{align*}
   &A_{1}(x)=x-1-a,  A_{2}(x)=x-1+a,\\
   &A_{3}(x)=x^2-(2-a)x+a^2-a+1, A_{4}(x)=x^2-2x+a^2+1,\\
   &A_{5}(x)=x^4-(4-a)x^3+(a^2-3a+6)x^2+(a^3-2a^2+3a+4)+a^4-a^3+a^2-a+1,\\
   &A_{6}(x)=x^2-(a+2)x+a^2+a+1.
 \end{align*}
\begin{lemma}\textup{(\cite{s2002})}\label{Molemma}
  Let $m,c$ be integers with $m>0$. Then we have
 \begin{equation*}
 \sum\limits_{d\mid m}\mu(\frac md)d\delta_{d\mid c}=\varphi(m)\frac{\mu(m/(c,m))}{\varphi(m/(c,m))},
 \end{equation*}
 where
 \[\delta_{d\mid c}=\begin{cases}1,&\mbox{ if  }d\mid c  \mbox{ holds};\\0,&\mbox{otherwise}.\end{cases} \]
 \end{lemma}
 \begin{proof}
  We can find that both sides are multiplicative with respect to $m$, thus we only need to prove it when $m$ is a prime power. For any prime $p$ and positive integer $k$, we have
 \begin{align*}
 \sum\limits_{d\mid p^k}\mu(\frac {p^k}d)d\delta_{d\mid c}
 &=\sum\limits_{s=0}^k\mu(p^{k-s})p^s\delta_{p^s\mid c}\\
 &=p^k\delta_{p^k\mid c}-p^{k-1}\delta_{p^{k-1}\mid c}\\
 &=\begin{cases}p^k-p^{k-1}&\textup{if}\; p^k\mid c,\\
 -p^{k-1}&\textup{if}\;p^{k-1}\parallel c,\\
 0&\textup{if} \;p^{k-1}\nmid c.\end{cases}\\
 &=\varphi(p^k)\frac{\mu(p^k/(c,p^k))}{\varphi(p^k/(c,p^k))}.
 \end{align*}
 This concludes the proof.
 \end{proof}

 \begin{theorem}\label{Wtheorem}
   Let $m,n\in\mathbb{Z}^+ ,r\in\mathbb{Z}$. Then
  \begin{equation*}
  W_{n}(r,m)=\varphi(m)\sum\limits_{k=0}^n\frac{\mu(m/(k-r,m))}{\varphi(m/(k-r,m))}\binom nka^k.
  \end{equation*}
 \end{theorem}
 \begin{proof}
   By Theorem \ref{Maintheorem}, Lemma \ref{Molemma}  and M$\ddot{\textup{o}}$bius Inversion Theorem, we have
   \begin{align*}
   W_{n}(r,m)
   &=\sum\limits_{d\mid m}\mu(\frac md)d\left[\begin{array}{c}n \\r\\\end{array}\right]_d(a)\\
   &=\sum\limits_{d\mid m}\mu(\frac md)d\sum\limits_{k=0}^n\binom nka^k\delta_{d\mid k-r} \\
   &= \sum\limits_{k=0}^n\binom nka^k\sum\limits_{d\mid m}\mu(\frac md)d\delta_{d\mid k-r}\\
   &=\varphi(m)\sum\limits_{k=0}^n\frac{\mu(m/(k-r,m))}{\varphi(m/(k-r,m))}\binom nka^k.
   \end{align*}
 \end{proof}

   \begin{corollary}\label{Wcorollary}
     Let $m,n$ be two  relatively prime positive integers. Then we have
      \begin{equation*}
      W_{n}(0,m)-\varphi(m)-\mu(m)a^n=\varphi(m)n\sum\limits_{k=1}^{n-1}\frac{\mu(m/(m,k))}{\varphi(m/(m,k))}\binom{n-1}{k-1}\frac{a^k}{k}
     \end{equation*}
     and
      \begin{equation*}
      W_{n}(n,m)-\varphi(m)a^n-\mu(m)=\varphi(m)n\sum\limits_{k=1}^{n-1}\binom{n-1}{k-1}\frac{\mu(m/(m,k))}{\varphi(m/(m,k))}\frac{a^{n-k}}{k}.
      \end{equation*}
   \end{corollary}
   \begin{proof}
     Since $\binom nk=\frac nk\binom{n-1}{k-1}$ for $1\leq k\leq n$, we can derive the results by setting $r=0,n$ respectively in Theorem \ref{Wtheorem},
   \end{proof}
   \begin{corollary}\label{Wpcorlllary}
      Let $m\in\mathbb{Z}^+$ and $p$ be an odd prime not dividing $am$. Then we have

      \begin{equation*}
      \frac{W_{p}(0,m)- \varphi(m)-\mu(m)a^p}{p}\equiv-\varphi(m)\sum\limits_{k=1}^{p-1}\frac{\mu(m/(m,k))}{\varphi(m/(m,k))}\cdot\frac{(-a)^k}{k}\pmod p,
      \end{equation*}
       and
       \begin{equation*}
      \frac{W_{p}(p,m)-\varphi(m)a^p-\mu(m)}{p}\equiv\varphi(m)\sum\limits_{k=1}^{p-1}\frac{\mu(m/(m,k))}{\varphi(m/(m,k))}\cdot\frac{1}{k(-a)^{k-1}}\pmod p.
      \end{equation*}
   \end{corollary}
   \begin{proof}
    Since $\binom{p-1}{k}=(-1)^k$ for $0\leq k\leq p-1$, the results follow from Corollary \ref{Wcorollary}.
   \end{proof}

 \section{Some New Congruences}
 In this section, we give some new congruences by using the results of \cite{dy}.
 \begin{lemma}\label{3uvlemma}
   Let $p\nmid 3a(2-a)(a^3+1)$ be and odd prime, and $\{u_n\}_{n\geq0},\{v_n\}_{n\geq0}$ be the Lucas sequences defined as
    $$u_0=0,\;u_1=1,\;u_{n+1}=(2-a)u_n-(a^2-a+1)u_{n-1}\;\textup{for}\;n\geq1;$$
    $$v_0=2,\;v_1=(2-a),\;v_{n+1}=(2-a)v_n-(a^2-a+1)v_{n-1}\;\textup{for}\;n\geq1.$$
    Then we have: \\
   \begin{description}
    \item[(1)]
    \[  \frac{u_p-\left(\frac{-3}{p}\right)}{p}\equiv\sum\limits_{k=1}^{\frac{p-1}2}\frac{(-3)^{k-1}}{2k-1}\cdot\left(\frac{a}{2-a}\right)^{2k-2}
    +\left(\frac{-3}{p}\right)\left(q_p(a)-q_p(2)+\frac12q_p(3)\right)\pmod p;\]
   \item[(2)]
    \[ \quad\frac{v_{p}-(2-a)}{p}\equiv
    (2-a)\left[-\frac12\sum\limits_{k=1}^{\frac{p-1}2}\frac{(-3)^k}{k}\cdot\left(\frac{a}{2-a}\right)^{2k}-q_p(2)+q_p(2-a)\right]\pmod p.\]
   \end{description}
 \end{lemma}
 \begin{proof}
 By Lemmas  2.1 and 2.2 of \cite{dy}, we have
 $u_p=\frac1{a\sqrt{-3}} \left[\left(\frac{2-a}{2}+\frac a2\sqrt{-3}\right)^p- \left(\frac{2-a}{2}-\frac a2\sqrt{-3}\right)^p\right]$
 ,  $v_p= \left(\frac{2-a}{2}+\frac a2\sqrt{-3}\right)^p+ \left(\frac{2-a}{2}-\frac a2\sqrt{-3}\right)^p$, and $u_p\equiv\left(\frac{-3}{p}\right)\pmod p$,  $v_p=(2-a)u_p-2(a^2-a+1)u_{p-1} =2u_{p+1}-(2-a)u_p\equiv(2-a)\pmod p$. Then
  \begin{align*}
   2^{p-1}u_p
   &=\sum_{\substack{k=0\\k\; odd}}^{p}\binom pk(2-a)^{p-k}(a\sqrt{-3})^{k-1}\\
   &=a^{p-1}(-3)^{\frac{p-1}{2}}+\sum\limits_{k=1}^{\frac{p-1}{2}}\binom p{2k-1}(2-a)^{p-2k+1}a^{2k-2}(-3)^{k-1}\\
   &=a^{p-1}(-3)^{\frac{p-1}{2}}+p\sum\limits_{k=1}^{\frac{p-1}{2}}\frac{(-3)^{k-1}}{2k-1}\binom {p-1}{2k-2}(2-a)^{p-2k+1}a^{2k-2}\\
   &\equiv a^{p-1}(-3)^{\frac{p-1}{2}}+ p\sum\limits_{k=1}^{\frac{p-1}{2}}\frac{(-3)^{k-1}}{2k-1}\cdot\left(\frac{a}{2-a}\right)^{2k-2}\pmod {p^2},\\
 \end{align*}
 and
 \begin{align*}
   2^{p-1}v_p
   &=\sum_{\substack{k=0\\k\; even}}^{p}\binom pk(2-a)^{p-k}(a\sqrt{-3})^k\\
   &=(2-a)^p+\sum\limits_{k=1}^{\frac{p-1}{2}}\binom p{2k}(2-a)^{p-2k}a^{2k}(-3)^k\\
   &=(2-a)^p+p\sum\limits_{k=1}^{\frac{p-1}{2}}\frac{(-3)^k}{2k}\binom {p-1}{2k-1}(2-a)^{p-2k}a^{2k}\\
   &\equiv(2-a)^p-\frac{2-a}2 p\sum\limits_{k=1}^{\frac{p-1}{2}}\frac{(-3)^k}{k}\cdot\left(\frac{a}{2-a}\right)^{2k}\pmod {p^2}.\\
 \end{align*}
Hence (1) and (2) follow from Lemma 2.6(1) of \cite{dy}.
 \end{proof}
 \begin{corollary}\label{30corollary}
    Let $p\nmid 3a(2-a)(a^3+1)$ be and odd prime. Then we have

    \[\sum\limits_{k=1}^{[\frac{p}{3}]}\frac{(-a)^{3k}}{k}\equiv(2-a)\left[\frac12\sum\limits_{k=1}^{\frac{p-1}2}\frac{(-3)^k}{k}\cdot\left(\frac{a}{2-a}\right)^{2k}
    +q_p(2)-q_p(2-a)\right]-(a+1)q_p(a+1)\pmod p.\]

 \end{corollary}
 \begin{proof}
    The result follows from   Lemma 4.9 of  \cite{dy} and Lemma \ref{3uvlemma}(2).
 \end{proof}
\begin{theorem}\label{31theorem}

Let $p\nmid 3a(a-1)(2-a)(a^3+1)$ be an odd prime, and $\{u_n\}_{n\geq0}$   be the Lucas sequence defined as
    $$u_0=0,\;u_1=1,\;u_{n+1}=(2-a)u_n-(a^2-a+1)u_{n-1}\;\textup{for}\;n\geq1.$$
\begin{description}
    \item[(1)] If $p\equiv1\pmod 3$,  we have
\[\frac{u_{p-1}}{p}\equiv-\frac{2}{a(a-1)}\sum\limits_{k=1}^{\frac{p-1}{3}}\frac{(-a)^{3k-1}}{3k-1}+\frac{a+1}{3a(a-1)}\left(q_p(a^2-a+1)-2q_p(a+1)\right)\pmod p\]
and
\begin{align*}
\sum\limits_{k=1}^{\frac{p-1}{3}}\frac{(-a)^{3k-1}}{3k-1}
&\equiv\frac{a(a-1)}{a-2}\sum\limits_{k-1}^{\frac{p-1}{2}}\frac{(-3)^{k-1}}{2k-1}\cdot\left(\frac{a}{2-a}\right)^{2k-2}\\
&+\frac{a(a-1)}{a-2} [q_p(a)-q_p(2)+\frac12q_p(3)]\\
&-\frac{1}{3}(a+1)q_p(a+1)- \frac{ a^2-a+1}{3(a-2)}q_p(a^2-a+1)  \pmod p.
\end{align*}
 \item[(2)] If $p\equiv2\pmod 3$,  we have
 \[\frac{u_{p+1}}{p}\equiv\frac{2(a^2-a+1)}{a(a-1)}\sum\limits_{k=1}^{\frac{p+1}{3}}\frac{(-a)^{3k-2}}{3k-2}-\frac{a^3+1}{3a(a-1)}\left(q_p(a^2-a+1)-2q_p(a+1)\right)\pmod p\]
 and
\begin{align*}
\sum\limits_{k=1}^{\frac{p+1}{3}}\frac{(-a)^{3k-2}}{3k-2}
&\equiv-\frac{a(a-1)}{a-2}\sum\limits_{k=1}^{\frac{p-1}{2}}\frac{(-3)^{k-1}}{2k-1}\cdot\left(\frac{a}{2-a}\right)^{2k-2}\\
&\quad+\frac{a(a-1)}{a-2} [q_p(a)-q_p(2)+\frac12q_p(3)]\\
&\quad-\frac{1}{3}(a+1)q_p(a+1)-\frac{ a^2-a+1}{3(a-2)}q_p(a^2-a+1)  \pmod p.
\end{align*}
\end{description}
\end{theorem}
\begin{proof}
Since$(2-a)^2-4(a^2-a+1)=-3a^2$, we have $p\mid u_{p-\left(\frac{-3}{p}\right)}$  by Lemma 2.2 of \cite{dy}.
Let $\{v_n\}_{n\geq0}$ be the Lucas sequence define as Lemma \ref{3uvlemma}.

(1) By Lemma 2.1 and  Theorem 4.1 of \cite{dy}, we have $-(a+1)u_p+(a^2-a+1)u_{p-1}=3\left[\begin{array}{c}p \\ 2\\\end{array}\right] _{3}(a)-(1+a)^p$ and $v_{p-1}=2u_p-(2-a)u_{p-1}$. Thus by   Lemma 2.4  of \cite{dy},  we have
 \begin{align*}
 3a(a-1)u_{p-1}
 &=6\left[\begin{array}{c}p \\ 2\\\end{array}\right] _{3}(a)-2(1+a)^p+(a+1)v_{p-1}\\
 &\equiv-6p\sum\limits_{k-1}^{\frac{p-1}{3}}\frac{(-a)^{3k-1}}{3k-1}+(a+1)\left[(v_p-2)-2((a+1)^{p-1}-1)\right]\pmod{p^2}
 \end{align*}
 and
 \begin{align*}
 3a(a-1)(u_{p}-1)
 &=3(2-a)\left[\begin{array}{c}p \\ 2\\\end{array}\right] _{3}(a)-(2-a)(1+a)^p+(a^2-a+1)v_{p-1}-3a(a-1)\\
 &\equiv-3(2-a)p\sum\limits_{k=1}^{\frac{p-1}{3}}\frac{(-a)^{3k-1}}{3k-1}-(2-a)(1+a)((a+1)^{p-1}-1)\\
 &\quad+(a^2-a+1)(v_{p-1}-2) \pmod{p^2}.
 \end{align*}
 Thence by Lemma 2.7  of \cite{dy} and Lemma \ref{3uvlemma}(1),
 \[\frac{u_{p-1}}{p}\equiv-\frac{2}{a(a-1)}\sum\limits_{k=1}^{\frac{p-1}{3}}\frac{(-a)^{3k-1}}{3k-1}+\frac{a+1}{3a(a-1)}\left(q_p(a^2-a+1)-2q_p(a+1)\right)\pmod p,\]
 and
\begin{align*}
\sum\limits_{k=1}^{\frac{p-1}{3}}\frac{(-a)^{3k-1}}{3k-1}
&\equiv\frac{a(a-1)}{a-2}\sum\limits_{k-1}^{\frac{p-1}{2}}\frac{(-3)^{k-1}}{2k-1}\cdot\left(\frac{a}{2-a}\right)^{2k-2}\\
&+\frac{a(a-1)}{a-2} [q_p(a)-q_p(2)+\frac12q_p(3)]\\
&-\frac{1}{3}(a+1)q_p(a+1)- \frac{ a^2-a+1}{3(a-2)}q_p(a^2-a+1)  \pmod p.
\end{align*}
(2) By Lemma 2.1 and  Theorem 4.1 of \cite{dy}, we have $-u_{p+1}+(a+1)u_{p}=3\left[\begin{array}{c}p \\ 1\\\end{array}\right] _{3}(a)-(1+a)^p$ and $v_{p+1}=(2-a)u_{p+1}-2(a^2-a+1)u_p$. Thus by  Lemmas 2.4  of \cite{dy}, we have
 \begin{align*}
 -3a(a-1) u_{p+1}
&=6(a^2-a+1)\left[\begin{array}{c}p \\ 1\\\end{array}\right] _{3}(a)-2(a^2-a+1)(1+a)^p+(a+1)v_{p+1}\\
 &\equiv-6(a^2-a+1)p\sum\limits_{k=1}^{\frac{p+1}{3}}\frac{(-a)^{3k-2}}{3k-2}-2(a^2-a+1)(a+1)\left[(a+1)^{p-1}-1\right]\\
 &\quad+(a+1)\left[v_{p+1}-2(a^2-a+1)\right]\pmod{p^2}
\end{align*}
 and
 \begin{align*}
 -3a(a-1)(u_{p}+1)
 &=3(2-a)\left[\begin{array}{c}p \\ 1\\\end{array}\right] _{3}(a)- (2-a)(1+a)^p+ v_{p+1}-3a(a-1)\\
 &\equiv-3 (2-a )p\sum\limits_{k=1}^{\frac{p+1}{3}}\frac{(-a)^{3k-2}}{3k-2}- (2-a)(a+1)\left[(a+1)^{p-1}-1\right]\\
 &\quad  +v_{p+1}-2(a^2-a+1) \pmod{p^2}.
 \end{align*}
 Thence by  Lemma 2.7  of \cite{dy}  and  Lemma \ref{3uvlemma}(1),
 \[\frac{u_{p+1}}{p}\equiv\frac{2(a^2-a+1)}{a(a-1)}\sum\limits_{k=1}^{\frac{p+1}{3}}\frac{(-a)^{3k-2}}{3k-2}-\frac{a^3+1}{3a(a-1)}\left(q_p(a^2-a+1)-2q_p(a+1)\right)\pmod p,\]
 and
\begin{align*}
\sum\limits_{k=1}^{\frac{p+1}{3}}\frac{(-a)^{3k-2}}{3k-2}
&\equiv-\frac{a(a-1)}{a-2}\sum\limits_{k=1}^{\frac{p-1}{2}}\frac{(-3)^{k-1}}{2k-1}\cdot\left(\frac{a}{2-a}\right)^{2k-2}\\
&+\frac{a(a-1)}{a-2} [q_p(a)-q_p(2)+\frac12q_p(3)]\\
&-\frac{1}{3}(a+1)q_p(a+1)-\frac{ a^2-a+1}{3(a-2)}q_p(a^2-a+1)  \pmod p.
\end{align*}
\end{proof}
Set $a=-2$ in Corollary \ref{30corollary} and Theorem \ref{31theorem}, we have the following two corollaries.
 \begin{corollary}
   Let $p\neq3,7$ be and odd prime. Then we have
   \begin{description}
    \item[(1)]
        \[\sum\limits_{k=1}^{[\frac p3]}\frac{8^k}{k}\equiv\sum\limits_{k=1}^{\frac {p-1}{2}}\frac{2}{k}\cdot\left(-\frac34\right)^k-4q_p(2)\pmod p.\]
    \item[(2)]
     If $p\equiv1 \pmod 3$,
               \[\sum\limits_{k=1}^{\frac{p-1}{3}}\frac{8^k}{3k-1}\equiv4\sum\limits_{k=1}^{\frac {p-1}{2}}\frac{1}{2k-1}\cdot\left(-\frac34\right)^k-\frac 32q_p(3)+\frac 76q_p(7)\pmod p.\]
    If $p\equiv2 \pmod 3$,
               \[\sum\limits_{k=1}^{\frac{p+1}{3}}\frac{8^k}{3k-2}\equiv-8\sum\limits_{k=1}^{\frac {p-1}{2}}\frac{1}{2k-1}\cdot\left(-\frac34\right)^k-3q_p(2)+\frac 73q_p(7)\pmod p.\]
   \end{description}

 \end{corollary}
 \begin{corollary}\label{-2corollary}
Let $p\neq3,7$ be and odd prime,  and $\{u_n\}_{n\geq0}$   be the Lucas sequence defined as
    $$u_0=0,\;u_1=1,\;u_{n+1}=4u_n-7u_{n-1}\;\textup{for}\;n\geq1.$$
       Then,   if $p\equiv1\pmod 3$,
\[\frac{u_{p-1}}p\equiv-\frac{1}{6}\sum\limits_{k=1}^{\frac{p-1}{3}}\frac{8^k}{3k-1}-\frac{1}{18} q_p(7)\pmod p,\]
if $p\equiv2\pmod 3$,
\[\frac{u_{p+1}}p\equiv\frac{7}{12}\sum\limits_{k=1}^{\frac{p+1}{3}}\frac{8^k}{3k-2}+\frac{7}{18} q_p(7)\pmod p.\]
\end{corollary}
 The following theorem can reduce the summation terms occurring in the expression of Lucas quotients in Corollary 4.11 of \cite{dy} and    Corollary 3.5.
 \begin{theorem}
    Let $p\neq3,7$ be an odd prime, and $\left\{u_n\right\}_{n\geq0}$ be the Lucas sequence defined as Corollary 3.5.
Then if $p\equiv1\pmod3$,
\begin{align*}
\frac{u_{p-1}}p
&\equiv\frac16\sum\limits_{k=1}^{\frac{p-1}{6}}\frac{64^k}{k}+\frac13q_p(7)+\frac12q_p(3)\\
&\equiv-\frac{1}{3}\sum\limits_{k=1}^{\frac{p-1}{6}}\frac{64^k}{6k-1}-\frac{1}{18} q_p(7)+\frac 16q_p(3)  \pmod p,
\end{align*}
if $p\equiv2\pmod3$,
\begin{align*}
\frac{u_{p+1}}p
&\equiv-\frac76\sum\limits_{k=1}^{\frac{p-5}{6}}\frac{64^k}{k}-\frac73q_p(7)-\frac72q_(3)\\
&\equiv\frac{7}{6}\sum\limits_{k=1}^{\frac{p+1}{6}}\frac{64^k}{6k-2}+\frac{7}{18} q_p(7)+\frac 76q_p(3)\pmod p.
\end{align*}
\end{theorem}

 \begin{proof}
      By  Lemma 2.4 and Theorem 4.5  of \cite{dy}, if $p\equiv1\pmod 3$,
  \[3^{p-1}-(-3)^{\frac{p-1}{2}}=\left[\begin{array}{c}p \\2 \\\end{array}\right] _3(2)\equiv-p\sum\limits_{k=1}^{\frac{p-1}{3}}\frac{(-2)^{3k-1}}{3k-1}\pmod{p^2},\]
  if $p\equiv2\pmod 3$,
  \[3^{p-1}+(-3)^{\frac{p-1}{2}}=\left[\begin{array}{c}p \\1 \\\end{array}\right] _3(2)\equiv-p\sum\limits_{k=1}^{\frac{p+1}{3}}\frac{(-2)^{3k-2}}{3k-2}\pmod{p^2}.\]
   Thus by the Lemma 2.6 of \cite{dy}, we have
  \[ \sum\limits_{k=1}^{\frac{p-1}{3}}\frac{(-8)^{k}}{3k-1}\equiv q_p(3)\pmod p \;\textup{if} \;p\equiv1 \pmod 3,\] and
  \[ \sum\limits_{k=1}^{\frac{p+1}{3}}\frac{(-8)^{k}}{3k-2}\equiv-2q_p(3) \pmod p \;\textup{if} \;p\equiv2 \pmod 3.\]
  Hence the results follow from     Corollaries  4.7 and  4.11 of \cite{dy} and Corollary \ref{-2corollary}.
 \end{proof}
 \section{A Specific Lucas Sequence}
 \noindent Let $A,B\in\mathbb{Z}$. The Lucas sequences $u_n=u_n(A,B)(n\in\mathbb{N})$ and $v_n=v_n(A,B)(n\in\mathbb{N})$ are defined by
 \[u_0=1,\;u_1=1,\; u_{n+1}=Bu_n-Au_{n-1}(n\geq1);\]
  \[v_0=2,\;v_1=B,\; v_{n+1}=Bv_n-Av_{n-1}(n\geq1).\]
Next  we give some properties of  the Lucas sequences with $A=5$ and $B=2$. We need some lemmas. Let $D=B^2-4A.$
 \begin{lemma}\label{ulucasmod}
    Let $p$ be an odd prime not dividing $DA$.
  \begin{description}
    \item[(1)]
     If $p\equiv1\pmod 4$, then $p\mid u_{\frac{p-1}{4}}$ if and only if $v_{\frac{p-1}{2}}\equiv2A^{\frac{p-1}{4}}\pmod p$ and
      $p\mid v_{\frac{p-1}{4}}$ if and only if $v_{\frac{p-1}{2}}\equiv-2A^{\frac{p-1}{4}}\pmod p$.
    \item[(2)] If $p\equiv3\pmod 4$, then $p\mid u_{\frac{p+1}{4}}$ if and only if $v_{\frac{p+1}{2}}\equiv2A^{\frac{p+1}{4}}\pmod p$ and
       $p\mid v_{\frac{p+1}{4}}$ if and only if $v_{\frac{p+1}{2}}\equiv-2A^{\frac{p+1}{4}}\pmod p$.
  \end{description}
 \end{lemma}
 \begin{proof}
  (1) and (2) follow from the fact that $v_{2n}=v_n^2-2A^n=Du_n^2+2A^n.$
 \end{proof}
 \begin{lemma}\textup{(\cite{sun4})}\label{uvlucasmod}
  Let $p$ be an odd prime and $A'$ be an integer such that $4A'\equiv B^2-4A\pmod p $. Let
  $u_n'=u_n(A',B),\;v_n'=v_n(A',B)$. Then we have
  \[
    u_{\frac{p+1}{2}}\equiv\frac12\left(\frac2p\right)v'_{\frac{p-1}{2}}\pmod p,\;u_{\frac{p-1}{2}}\equiv-\left(\frac2p\right)u'_{\frac{p-1}{2}}\pmod p,
   \]
  \[v_{\frac{p+1}{2}}\equiv\left(\frac2p\right)v'_{\frac{p+1}{2}}\pmod p,\;v_{\frac{p-1}{2}}\equiv2\left(\frac2p\right)u'_{\frac{p+1}{2}}\pmod p.
  \]
\end{lemma}
  \begin{remark}
\begin{description}
    \item[(1)] Let $S_n=u_n(1,4),\;T_n=v_n(1,4)$. For any  prime $p>3$, by the facts that $u'_n=u_n(3,4)=\frac{1}{2}(3^n-1)$ and $v_n'=v_n(3,4)=3^n+1$,  we have
\begin{align*}
&S_{\frac{p+1}{2}}\equiv\frac12\left(\frac{2}{p}\right)\left[\left(\frac{3}{p}\right)+1\right]\pmod p,\quad
S_{\frac{p-1}{2}}\equiv-\frac12\left(\frac{2}{p}\right)\left[\left(\frac{3}{p}\right)-1\right]\pmod p,\\
&T_{\frac{p+1}{2}}\equiv \left(\frac{2}{p}\right)\left[3\left(\frac{3}{p}\right)+1\right]\pmod p,\quad
T_{\frac{p-1}{2}}\equiv\left(\frac{2}{p}\right)\left[3\left(\frac{3}{p}\right)-1\right]\pmod p.
\end{align*}

Thus by Lemma \ref{ulucasmod},   $p\mid S_{[\frac{p+1}{4}]}$ iff  $p\equiv1,19\pmod {24}$ and
$p\mid T_{[\frac{p+1}{4}]}$  iff  $p\equiv7,13\pmod {24}$. Sun  \cite{s2002} got these by studying the sum (\ref{generalsum}) for $a=1$ and $m=12$;

\item[(2)] Let $P_n=u_n(-1,2),\;Q_n=v_n(-1,2)$ and $u_n'=u_n(2,2),\;v_n=v_n'(2,2).$
For any odd prime $p$, by the facts that $u_{4n}'=0,u_{4n+1}'=(-4)^n,u_{4n+2}'=u_{4n+3}'=2(-4)^n$ and
$v_{4n}'=v_{4n+1}'=2(-4)^n,u_{4n+2}'=0,u_{4n+3}'=(-4)^{(n+1)}$, we have
\[P_{\frac{p-\left(\frac2p\right)}{2}}\equiv\begin{cases}0\pmod p,&  \textup{if}\,p\equiv1\pmod 4,\\(-1)^{[\frac{p+5}{8}]}2^{\frac{p-3}{4}}\pmod p,& \textup{if}\;p\equiv3\pmod 4,\end{cases}
\]

\[Q_{\frac{p-\left(\frac2p\right)}{2}}\equiv\begin{cases}(-1)^{[\frac{p}{8}]}2^{\frac{p+3}{4}}\pmod p,&   \textup{if}\;p\equiv1\pmod 4,\\0\pmod p,& \textup{if}\;p\equiv3\pmod 4,\end{cases}\]
and
\[P_{\frac{p+\left(\frac2p\right)}{2}}\equiv(-1)^{[\frac{p+1}{8}]}2^{[\frac{p }{4}]}\pmod p,\quad Q_{\frac{p+\left(\frac2p\right)}{2}}\equiv(-1)^{[\frac{p+5}{8}]}2^{[\frac{p+5}{4}]}\pmod p.\]
\end{description}
Sun got \cite{sun2} these by studying the sum (\ref{generalsum}) for $a=1$ and $m=8$.
\end{remark}

\begin{lemma}\label{uv1lucasmod}
   Let $p\nmid B$ be an odd prime and $A'$ be an integer such that $A'\equiv \frac{A}{B^2}\pmod p $. Let
  $u_n'=u_n(A',1),\;v_n'=v_n(A',1).$ Then we have
\[u_{\frac{p+1}{2}}\equiv\left(\frac Bp\right)u'_{\frac{p+1}{2}}\pmod p,\quad u_{\frac{p-1}{2}}\equiv\frac1B\left(\frac Bp\right)u_{\frac{p-1}{2}}'\pmod p,\]
\[v_{\frac{p+1}{2}}\equiv B\left(\frac Bp\right)v'_{\frac{p+1}{2}} \pmod p,\quad v_{\frac{p-1}{2}}\equiv \left(\frac Bp\right)v_{\frac{p-1}{2}}'\pmod p.\]
\end{lemma}
\begin{proof}
  By Lemma 2.1 of \cite{dy} and $D'=1-4A'\equiv\frac{D}{B^2}\pmod p,$ we have
   \begin{align*}
   u_{n}
    &=2\sum_{\substack{k=0\\k\;odd}}^{n}\binom{n}{k}\left(\frac{B}2\right)^{n-k}\left(\frac{D}2\right)^{\frac{k-1}{2}}\\
    &=2B^{n-1}\sum_{\substack{k=0\\k\;odd}}^{n}\binom{n}{k}\left(\frac{1}2\right)^{\frac{p-1}{2}-k}\left(\frac{D}{2B^2}\right)^{\frac{k-1}{2}}\\
    &\equiv2B^{n-1}\sum_{\substack{k=0\\k\;odd}}^{n}\binom{n}{k}\left(\frac{1}2\right)^{n-k}\left(\frac{D'}{2}\right)^{\frac{k-1}{2}}\\
    &=B^{n-1}u'_n\pmod p,
  \end{align*}
and
   \begin{align*}
   v_{n}
    &=2\sum_{\substack{k=0\\k\;even}}^{n}\binom{n}{k}\left(\frac{B}2\right)^{n-k}\left(\frac{D}2\right)^{\frac{k}{2}}\\
    &=2B^{n}\sum_{\substack{k=0\\k\;even}}^{n}\binom{n}{k}\left(\frac{1}2\right)^{\frac{p-1}{2}-k}\left(\frac{D}{2B^2}\right)^{\frac{k}{2}}\\
    &\equiv2B^{n}\sum_{\substack{k=0\\k\;even}}^{n}\binom{n}{k}\left(\frac{1}2\right)^{n-k}\left(\frac{D'}{2}\right)^{\frac{k}{2}}\\
    &=B^{n}v'_n\pmod p,
  \end{align*}
Thus 
 \begin{align*}
 &u_{\frac{p+1}{2}}\equiv B^{\frac{p-1}{2}}u'_{\frac{p+1}{2}}\equiv\left(\frac Bp\right)u'_{\frac{p+1}{2}}\pmod p,\\
 & u_{\frac{p-1}{2}}\equiv B^{\frac{p-3}{2}}u'_{\frac{p-1}{2}}\equiv\frac1B\left(\frac Bp\right)u_{\frac{p-1}{2}}'\pmod p, \\
 &v_{\frac{p+1}{2}}\equiv B^{\frac{p+1}{2}}v'_{\frac{p+1}{2}}\equiv B\left(\frac Bp\right)v'_{\frac{p+1}{2}} \pmod p,\\
 & v_{\frac{p-1}{2}}\equiv B^{\frac{p-1}{2}}v'_{\frac{p-1}{2}}\equiv \left(\frac Bp\right)v_{\frac{p-1}{2}}'\pmod p.
 \end{align*}
\end{proof}

\begin{theorem} \label{52lucas}
   Let $p\neq5$ be an odd prime and $\{U_n\}_{n\geq0}$ and $\{V_n\}_{n\geq0}$ be the Lucas sequences defined  as
$$U_0=0,U_1=1,U_{n+1} =2U_{n}-5U_{n-1} \;\textup{for}\;n\geq1;$$
$$V_0=2,V_1=2,V_{n+1}=2V_{n}-5V_{n-1}\;\textup{for}\;n\geq1.$$
\begin{description}
 \item[(1)] If $p\equiv\pm1\pmod 5$, we have
\begin{align*}
&U_{\frac{p+\left(\frac{-1}{p}\right)}{2}}\equiv\left(\frac{-1}{p}\right)(-1)^{[\frac{p+5}{10}]}5^{[\frac{p}{4}]}\pmod p,\\ &U_{\frac{p-\left(\frac{-1}{p}\right)}{2}}\equiv0\pmod p,\\
&V_{\frac{p+\left(\frac{-1}{p}\right)}{2}}\equiv2(-1)^{[\frac{p+5}{10}]}5^{[\frac{p}{4}]} \pmod p,\\
&V_{\frac{p-\left(\frac{-1}{p}\right)}{2}}\equiv2(-1)^{[\frac{p+5}{10}]}5^{[\frac{p+1}{4}]}\pmod p.
\end{align*}
\item[(2)] If $p\equiv\pm2\pmod 5$, we have
\begin{align*}
&U_{\frac{p+\left(\frac{-1}{p}\right)}{2}}\equiv\frac12\left(\frac{-1}{p}\right)(-1)^{[\frac{p+5}{10}]}5^{[\frac{p}{4}]}\pmod p,\\ &U_{\frac{p-\left(\frac{-1}{p}\right)}{2}}\equiv\frac12\left(\frac{-1}{p}\right)(-1)^{[\frac{p+5}{10}]}5^{[\frac{p+1}{4}]}\pmod p,\\
&V_{\frac{p+\left(\frac{-1}{p}\right)}{2}}\equiv4(-1)^{[\frac{p-5}{10}]}5^{[\frac{p}{4}]} \pmod p,\\
&V_{\frac{p-\left(\frac{-1}{p}\right)}{2}}\equiv0\pmod p.
\end{align*}
\end{description}
\end{theorem}
\begin{proof}
  Let $F_n=u_n(-1,1)$ and $L_n=v_n(-1,1)$ be Fibonacci sequence and its companion. Then by Lemmas \ref{uvlucasmod} and \ref{uv1lucasmod}, we have
\[U_{\frac{p+1}{2}}\equiv\frac12L_{\frac{p-1}{2}}\pmod p, \quad
U_{\frac{p-1}{2}}\equiv-\frac12F_{\frac{p-1}{2}}\pmod p,\]
\[V_{\frac{p+1}{2}}\equiv2L_{\frac{p+1}{2}}\pmod p,\quad
V_{\frac{p-1}{2}}\equiv2F_{\frac{p+1}{2}}\pmod p.\]
Thus by   Corollaries 1 and  2 of \cite{ss}, we can derive the results.
\end{proof}
\begin{remark}
 In \cite{dy}, we gave some congruences for the Lucas quotient $ U_{p-\left(\frac{-1}p\right)}/p$ by studying the sum (\ref{generalsum}) for $a=-2$ and $m=4$.
\end{remark}
\begin{corollary}
Let $p\neq5$ be an odd prime, $\{U_n\}_{n\geq0}$ and   $\{V_n\}_{n\geq0}$  be Lucas sequences defined as above.
 \begin{description}
 \item[(1)] If $p\equiv1\pmod4$, then $p\mid U_{\frac{p-1}{4}}$ if and only if $p\equiv1\pmod {20}$ and $p\mid V_{\frac{p-1}{4}}$ if and only if $p\equiv9\pmod {20}$.
  \item[(2)] If $p\equiv3\pmod4$, then $p\mid U_{\frac{p+1}{4}}$ if and only if $p\equiv19\pmod {20}$ and $p\mid V_{\frac{p+1}{4}}$ if and only if $p\equiv11\pmod {20}$.
 \end{description}
\end{corollary}
\begin{proof}
 (1) and (2) follow from Lemma \ref{ulucasmod}   and   Theorem \ref{52lucas}.
\end{proof}

\vspace{0.5cm}

\noindent \textbf{Acknowledgments}\quad The work of this paper
was supported by the NNSF of China (Grant No. 11471314), and the National Center for Mathematics and Interdisciplinary Sciences, CAS.

\end{document}